\documentclass[a4paper]{article}
\usepackage{amssymb}
\usepackage{amsmath}
\usepackage{amsthm}
\usepackage{authblk}
\usepackage[USenglish]{babel}
\usepackage{esint}
\usepackage{graphicx}
\usepackage[dvipsnames]{xcolor}
\usepackage[colorlinks=true]{hyperref}
\usepackage[latin1]{inputenc}
\newtheorem{teo}{Theorem}[section]
\newtheorem{lemma}[teo]{Lemma}
\newtheorem{prop}[teo]{Proposition}
\newtheorem{cor}[teo]{Corollary}
\newtheorem{rem}[teo]{Remark}

\newtheorem{ese}[teo]{Example}
\numberwithin{equation}{section}

\newcommand\R{\mathbb R}

\newcommand\mbb\mathbb
\newcommand\mbf\mathbf
\newcommand\mcal\mathcal
\newcommand\mfrak\mathfrak
\renewcommand\mit\mathit
\newcommand\mrm\mathrm
\newcommand\msf\mathsf
\renewcommand\a\alpha
\renewcommand\b\beta
\newcommand\g\gamma
\newcommand\G\Gamma
\renewcommand\d\delta
\newcommand\D\Delta
\newcommand\e\varepsilon
\newcommand\z\zeta
\renewcommand\t\theta
\newcommand\Th\Theta
\newcommand\la\lambda
\newcommand\La\Lambda
\newcommand\s\sigma
\newcommand\si\varsigma
\newcommand\Si\Sigma
\newcommand\ups\upsilon
\newcommand\U\Upsilon
\newcommand\ph\varphi
\renewcommand\o\omega
\renewcommand\O\Omega
\newcommand\wt\widetilde
\newcommand\wh\widehat
\newcommand\ol\overline
\newcommand\ul\underline
\newcommand\mr\mathring
\newcommand\ub\underbrace
\newcommand\pa\partial
\newcommand\n\nabla
\newcommand\fa\forall
\newcommand\ex\exists
\newcommand\es\emptyset
\newcommand\wk\rightharpoonup
\newcommand\inc\hookrightarrow
\newcommand\linf\varliminf
\newcommand\lsup\varlimsup
\newcommand\os\overset
\newcommand\us\underset
\newcommand\sr\stackrel
\newcommand\Ot\Leftarrow
\newcommand\To\Rightarrow
\newcommand\map\mapsto
\newcommand\ot\leftarrow
\newcommand\lot\longleftarrow
\newcommand\lto\longrightarrow
\newcommand\tot\leftrightarrow
\newcommand\ltot\longleftrightarrow
\newcommand\sm\setminus
\renewcommand\Cup\bigcup
\renewcommand\Cap\bigcap
\newcommand\sub\subset
\newcommand\Sub\Subset
\newcommand\sne\subsetneq
\newcommand\bus\supset
\newcommand\Bus\Supset

\newcommand\eq\equiv
\newcommand\ox\otimes
\newcommand\Ox\bigotimes
\newcommand\pl\oplus
\newcommand\Pl\bigoplus
\newcommand\x\times
\renewcommand\c\circ
\newcommand\q\quad
\renewcommand\l\left
\renewcommand\r\right
\newcommand\fr\frac

\allowdisplaybreaks

\def\sideremark#1{\ifvmode\leavevmode\fi\vadjust{\vbox to0pt{\vss
 \text to 0pt{\hskip\hsize\hskip1em
 \vbox{\hsize2.1cm\tiny\raggedright\pretolerance10000
 \noindent #1\hfill}\hss}\vbox to15pt{\vfil}\vss}}}%

\begin{document}

\everymath{\displaystyle}

\title{Asymptotic behavior of minimal solutions of $-\D u=\la f(u)$ as $\la\to-\infty$.}
\author{Luca Battaglia\thanks{Universit\`a degli Studi Roma Tre, Dipartimento di Matematica e Fisica, Largo S. Leonardo Murialdo 1, 00146 Roma - lbattaglia@mat.uniroma3.it}, Francesca Gladiali\thanks{Universit\`a degli Studi di Sassari, Dipartimento di Chimica e Farmacia, Via Piandanna 4, 00710 Sassari - fgladiali@uniss.it}, Massimo Grossi\thanks{Sapienza Universit\`a di Roma, Dipartimento di Matematica, Piazzale Aldo Moro 5, 00185 Roma - grossi@mat.uniroma1.it}}
\date{}

\maketitle

\begin{abstract}
\noindent
We consider the following Dirichlet problem
\begin{equation}\label{pfl}
\l\{\begin{array}{ll}-\D u=\la f(u)&\text{in }\O\\u=0&\text{on }\pa\O\end{array}\r.,\tag{$\mcal P_f^\la$}
\end{equation}
with $\la<0$ and $f$ non-negative and non-decreasing.\\
We show existence and uniqueness of solutions $u_\la$ for any $\la$ and discuss their asymptotic behavior as $\la\to-\infty$. In the expansion of $u_\la$ \emph{large solutions} naturally appear.
\end{abstract}

\

\section{Introduction and main results}\

In this paper we consider the problem 
\begin{equation}
\l\{\begin{array}{ll}-\D u=\la f(u)&\text{in }\O\\u=0&\text{on }\pa\O\end{array}\r.,\tag{$\mcal P_f^\la$}
\end{equation}
where $\la$ is a real parameter, $\O$ is a smooth bounded domain of $\R^N$ with $N\ge2$ and $f$ is a real function satisfying the following assumption,
\begin{equation}\label{ipo}
f\text{ non-decreasing},\q\q\q f(0)>0,\q\q\q f|_{0<f<f(0)}\text{ is }C^1.
\end{equation}
In this setting Crandall and Rabinowitz (\cite{cr}, Section $4$, see also \cite{mp}) prove for $0<\la<\la^*$ and $f$ convex the existence of a branch $u_\la$ of $stable$ positive solutions, i.e. satisfying
$$\la_1\big(-\D-\la f'(u_\la)I\big)>0,$$
(here $\la_1$ demotes the first eigenvalue with zero Dirichlet boundary conditions). This branch is \emph{minimal}, in the sense that any other solution $u$ of \eqref{pfl} verifies $u\ge u_\la$. There is a huge literature about minimal, non-minimal and stable solutions to \eqref{pfl}, see \cite{dupaigne-book} as an example. We just recall some results about two nonlinearities which play an important role in this paper:
\begin{itemize}
\item $f(t)=\l((t-t_0)^+\r)^p$ with $p\ge1$ and $t_0<0$ (\emph{Problem of confined plasma}). A lot of authors studied this problem (\cite{bm,bb,cf,t}) where the set $\{u>t_0\}$ is the $plasma$ and the set $\{u<t_0\}$ is the $vacuum$.
\item $f(t)=e^t$ (\emph{the Liouville equation}). There is a very large literature mainly when $\O\subset\R^2$, see for instance \cite{bp,egp,gg,s}. Much less is known in higher dimensions (\cite{jl}). Observe that in this case the function $v=-u$ solves $-\D v=-\la e^{-v}$, an equation which has been derived in \cite{gl} in the study of the stationary states for a model of evolution of the electronic density in the plasma (see also \cite{clmp1,clmp2}).
\end{itemize}

The aim of this paper is to complete the study of the branch of stable solutions by considering the case $\la<0$. Of course in this event by the maximum principle we get that $u<0$ in $\O$.\\ 
Quite surprisingly, this case was not considered in the literature and we will see that some new and interesting phenomena occur. It is easy to show that for any $\la<0$ there exists a unique solution $u_\la$ to \eqref{pfl}. So the interesting problem is to study the asymptotic behavior of $u_\la$ as $\la\to-\infty$.\\
In order to state our first result let us introduce the following number:
\begin{equation}\label{t0}
t_0:=\inf\{t\in(-\infty,0):\,f(t)>0\}\in[-\infty,0).
\end{equation}
We observe that $t_0$ is the same number which appears in the \emph{plasma problem} and $t_0=-\infty$ in the Liouville equation.
Next theorem gives a description of the solution to \eqref{pfl} for any $f$ verifying \eqref{ipo}. 
\begin{teo}\label{casogen}$ $\\
Assume $f$ satisfies \eqref{ipo}. Then, for any $\la<0$, \eqref{pfl} has a unique stable solution $u_\la$.\\
Moreover, $t_0<u_\la(x)<0$ for any $x\in\O$, where $t_0$ is defined by \eqref{t0}, and 
\begin{equation}\label{b1}
u_\la(x)\us{\la\to-\infty}\to t_0\text{ in }L^\infty_{\mrm{loc}}(\O).
\end{equation}
\end{teo}
By the definition of $f$ we have that if $f>0$ then $t_0=-\infty$ and so there is a \emph{full blow-up} of the solution $u_\la$ in $\O$ (these is the case of $f(t)=e^t$). On the other hand, for solutions of the confined plasma problem described before, we get that $u_\la(x)\us{\la\to-\infty}\to t_0$ in $\O$.\\
This means that when $\la$ is \emph{negative}, i.e. we have negative pressure, there is no \emph{vacuum} in $\O$ and so no \emph{free boundary} appears (see \cite{cf}).\\

An interesting property of the solution $u_\la$ for $t_0\in\R$ is the following (see Proposition \ref{lem-conv})
$$\la f(u_\la)\to0\q\text{as }\la\to-\infty.$$
However note that this is not true if $t_0=-\infty$ (see Example \ref{ese}).\\
The result in \eqref{b1} is not surprising if one looks at the functional 
\begin{equation}\label{J}
J_\la(v):=\frac 12 \int_{\O}|\nabla v|^2 dx-\la \int_\O F(v)dx\end{equation}
associated with \eqref{pfl} for $F(s)=\int _{t_0}^s f(t) dt$.
It is easy to see that $J_\la$ is coercive and $u_\la$ is the minimum. The presence of the positive constant $-\la$ in front of the potential term suggests that it is convenient for $u_\la$ to minimize it, i.e. to reach the value $t_0$ even if this increases the kinetic term which becomes infinite near the boundary. Indeed in the examples \ref{ese2} (case $ii)$ and \ref{ese} (case $i)$ we find that $J_\la(u_\la)=(C+o(1))\sqrt{-\la}$ as $\la\to -\infty$ and both the kinetic and the potential term gives a contribution of order $\sqrt{-\la}$.\\
This phenomenon has some similarities with the Ginzburg-Landau problem
\[J_\e(v):=\frac 12 \int_{\O}|\nabla v|^2 dx+\frac 1{\e^2} \int_\O F(v)dx\]
where  $f$ is a double well potential and the minimizers $u_\e$ are characterized by a phase transition among the two zeroes of the potential term, say $\pm1$. \\
In our case, obviously, there is not phase transition, since $f$ has the geometry of a single well and indeed Theorem \ref{casogen} says that $u_\la\to t_0   \chi_\O$ in $L^{\infty}_{\mrm{loc}}$. \\
Nevertheless, since in our case $u_\la=0$ on the boundary, we think that some of the characteristics of the double well potential should appear near $\partial \Omega$.\\
Since $t_0<v<0$ a simple observation and  the coarea formula gives
\[\begin{split}
J_\la(v)& \geq \sqrt {-\la}\int_\O \sqrt{2F(v)}|\nabla v| dx
=\sqrt {-\la}\int_{t_0}^0 \left( \int _{\O\cap\{v=s\}} \sqrt{2F(s)} d\mathcal{H}^{n-1}(y)\right)ds\\
&=\sqrt {-\la}\int_{t_0}^0  \sqrt{2F(s)} \mathcal{H}^{n-1}(\{v=s\})ds
\end{split}\]
where $ \mathcal{H}^{n-1}$ is the $n-1$-dimensional Hausdorff measure. Hence if $u_\la$  minimizes $J_\la$ it is natural to expect that, far from the boundary $u_\la \to t_0$, which is the unique  zero of the potential $F(v)$. It is likewise natural to expect that near the boundary the solution $u_\la$ should minimize $\mathcal{H}^{n-1}(\{v=s\})$ i.e. the level sets are of minimal perimeter among the ones that converges to $\partial \O$. \\
What it should be natural to expect is that $u_\la(x)=u_\la\left( d(x,\partial \Omega)\right)$ where $d(x,\partial \Omega)$ is the distance of the point $x$ from the boundary which is what happen for the double well potential.

Next aim is to improve Theorem \ref{casogen} computing a more detailed asymptotic behavior of the expansion \eqref{b1}. \\
Even if our analysis in this paper does not allow to obtain information near $\partial \O$ what we will see is that all solutions $v$ of the limit problems which arise in the refined study of $u_\la$ as $\la\to \infty$ have this nice geometrical property that near the boundary $v(x)=v\left( d(x,\partial \Omega)\right)$.\\
As expected the value of $t_0$ and the shape of $f$ will play a crucial role. For this reason we will state our results by separating the case in which $t_0$ is finite from that where $t_0=-\infty$.

\subsection{The case $t_0\in\R$}\

In this case from Theorem \ref{casogen} we have that the solution $u_\la\to t_0$ in $\O$. The aim of this section is to improve \eqref{b1} computing the additional terms of the expansion.\\
The model nonlinearity is $f(t)=\l((t-t_0)^+\r)^p$ with $p\ge0$ and $t_0$ negative.\\
As remarked before this problem was studied by many people as $\la>0$ and $p\ge1$. For $p>1$ we just recall \cite{bm} and the references therein and if $p=1$ we mention \cite{t,bb,bs,cf}. In this last paper it was also studied the asymptotic behavior of the solution as $\la\to+\infty$. In this case the region occupied by the plasma, namely $\{x\in\O\text{ such that }u_\la>-t_0\}$ has diameter converging to $0$. We will see that as $\la$ is negative the opposite phenomenon occurs. On the other hand, as $\la\to-\infty$, $p=1$ is a threshold for our problem where the behavior changes dramatically. In particular, if $p>1$ \emph{large solutions} $v$ appear in the expansion of the solution. Let us recall that $v$ is a large solution in $\O$ if it satisfies
\begin{equation}\label{large}
\l\{\begin{array}{ll}\D v=g(v)&\text{in }\O\\v(x)\us{x\to\pa\O}\to+\infty\end{array}\r.
\end{equation}
There is a massive literature about existence, uniqueness and asymptotic analysis of solutions $v$ to \eqref{large}, so it is impossible to give en exhaustive list of references. We just recall the seminal papers by Keller \cite{kel} and Osserman \cite{oss} where it was proved that if $g$ is a positive, continuous, non-decreasing function then \eqref{large} admits solutions if and only if the following \emph{Keller-Osserman} condition holds:
\begin{equation}\label{keller}
\int_{t_1}^{+\infty}\frac{\mrm dt}{\sqrt{\int_{t_1}^tg(s)\mrm ds}}<+\infty,
\end{equation}
for some $t_1>0$. The uniqueness of large solutions has been established under some additional assumptions on $f$ and the regularity of the domain $\O$ (see \cite{ddgr} for references and new results). Here we quote the result in \cite{bmJAM} where the authors proved the uniqueness of the large solution to \eqref{large} when $g(t)=t^p$ and $p>1$ and \cite{mv} in the case when $g(t)=e^t$.\\
Now we are in position to state our result.

\begin{teo}\label{finito}$ $\\
Let $u_\la$ be the unique solution to \eqref{pfl}. Assume there exists some $\g(\a)\us{\a\to0^+}\to0$ such that
\begin{equation}\label{g}
g_0(t)\le\frac{f(\a t+t_0)}{\g(\a)}\us{\a\to0^+}\to t^p,\q\text{with $p\ge0$ locally uniformly in }t>0,
\end{equation}
for some $g_0$ satisfying \eqref{keller}
Then the following alternative holds:
\begin{itemize}
\item[(i)] If $\frac{\g(\a)}\a\us{\a\to0^+}\to0$, then
$$u_\la=t_0+\a_\la\big(v+o(1)\big)\q\text{as }\la\to-\infty\ \text{in }C^1_{\mrm{loc}}(\O),$$
where $v$ is the unique \emph{large solution} to \eqref{large} with $g(t)=t^p$, for some $\a_\la\us{\la\to-\infty}\to0^+$.

\item[(ii)] If $\frac{\g(\a)}\a\us{\a\to0^+}{\not\to}0$ and in addition:
\begin{itemize}
\item either $\O$ is a ball,
\item or $\O\sub\R^2$ and/or $\O$ is strictly convex and
\begin{equation}\label{i2}
\frac{f(\a t+t_0)}{\g(\a)}\le Ct^q\hbox{ for some }C>0,0\le q\le1,t>0;
\end{equation}
\end{itemize}
then,
$$u_\la(x)=t_0+\a_\la\l(v\l(\frac{x-x_\la}{\e_\la}\r)+o(1)\r)\q\text{as }\la\to-\infty\ \text{in }C^1_{\mrm{loc}}(\R^N),$$
where $v$ is an entire solution to 
\begin{equation}\label{soluzione-intera}
\l\{\begin{array}{ll}\D v=v^p&\text{in }\R^N\\v\ge v(0)=1.\end{array}\r.,
\end{equation}
\end{itemize}
for some $\a_\la,\e_\la\us{\la\to-\infty}\to0$ and $x_\la$ being a minimum point of $u_\la$.
\end{teo}
Even if the convergence in Theorem \ref{finito} does not allow to obtain information on the behavior of $u_\la$ near the boundary of $\O$, we observe here that the large solution to \eqref{large} satisfies 
\[ \lim_{x\to x_0} \frac{\psi(u(x))}{d(x,\partial \O)}=1,\]
where $x_0\in \partial \O$. Here $\psi$ is a function which depends only on the nonlinear term $g$ in \eqref{large}, see \cite{bmJAM}.

\begin{rem}\label{gnn}$ $\\
The assumption that $\O$ is planar or strictly convex is essential to have $x_\la\us{\la\to-\infty}\to x_0\in\O$, as in in the paper \cite{gnn} by Gidas, Ni and Nirenberg (see Corollary 3 and the Problem stated just below it). In the case of a ball, one does not even need to assume \eqref{i2}, essentially because all solutions are radial.\\
Notice that if $\O$ is a ball, then the solution of \eqref{soluzione-intera} is also radially symmetric. So it is uniquely determined as the solution to the O.D.E. 
$$\l\{\begin{array}{l}v''(r)+\frac{N-1}r v'(r)=v(r)^p\quad\hbox{in }\R\\v'(0)=0\\v(0)=1.\end{array}\r.$$
\end{rem}

\begin{rem}$ $\\
The assumption  $\frac{f(\a t+t_0)}{\g(\a)}\us{\a\to0^+}\to t^p$ in \eqref{g} is rather general and it is satisfied when the nonlinearity $f$ decay at zero at $t_0$ as a power or slowlier. Indeed it is equivalent to ask that
$$\frac{f(\a t+t_0)}{\g(\a)}\us{\a\to 0}\to g(t)\q\text{ for some $g$}.$$
See Lemma \ref{omog} for details. Observe that the condition that $\frac{f(\a t+t_0)}{\g(\a)}$ is bounded from below by $g_0$ in \eqref{g} is needed only in case $(i)$.\\
When, instead, the nonlinearity $f$ decay at zero faster, as in the case of 
$$f(t)=\l\{\begin{array}{ll}0&\text{ for }t<t_0\\e^{-\frac1{t-t_0}}&\text{ for }t_0<t<0\end{array}\r.$$
we still have that a \emph{large solution} appears in the expansion of $u_\la$. However we have to modify \eqref{g} assuming there exists $\a(\b)\ge0$ such that
\begin{equation}\label{g2}
\begin{split}
&\a(\b)\us{\b\searrow t_0}\to0,\q\q\frac{f(\b)}{\a(\b)}\us{\b\searrow t_0}\to0\\
&\text{and}\\
& g_0(t)\le\frac{f(\a(\b)t+\b)}{f(\b)}\us{\b\searrow t_0}\to g(t)\text{ locally uniformly for }t>-\sup_\b\frac{\b-t_0}{\a(\b)},
\end{split}
\end{equation}
for some $g_0$ satisfying \eqref{keller}, and we get $$u_\la(x)=\b_\la+\a_\la\big(v+o(1)\big)\q\text{as }\la\to-\infty\ \text{in }C^1_{\mrm{loc}}(\O),$$
where $v$ is the \emph{large solution} to \eqref{large}, corresponding to $g(t)$.
In this case an exponential function $g(t)$ can appear in the limit problem.
\end{rem}\

Due to the important role played by the nonlinearity $f(t)=\l((t-t_0)^+\r)^p$ we would like to state Theorem \ref{finito} expressly for this case. Note that $p>1$ corresponds to the case $(i)$ in Theorem \ref{finito} and $p\le1$ to $(ii)$.
\begin{cor}\label{i3}$ $\\
Let $\la<0$ and $u_\la$ be the unique negative solution to
$$\l\{\begin{array}{ll}-\D u=\la\l((u-t_0)^+\r)^p&\text{in }\O\\
u=0&\text{on }\pa\O.\end{array}\r.$$
Then the following alternative holds:
\begin{itemize}
\item[(i)]If $p>1$, then
$$u_\la=t_0+\frac{v+o(1)}{(-\la)^\frac1{p-1}}\q\text{as }\la\to-\infty\text{ in }C^\infty_{\mrm{loc}}(\O),$$
where $v$ is the unique positive solution to
$$\l\{\begin{array}{ll}\D v=v^p&\text{in }\O\\v(x)\us{x\to\pa\O}\to+\infty\end{array}\r.;$$

\item[(ii)]If $0\le p\le1$ and $\O$ is either planar or strictly convex, then setting $\a_\la=u_\la(x_\la)-t_0$ and $\e_\la=\sqrt{\frac{\a_\la^{1-p}}{-\la}}$ we have that
$$u_\la\l(\e_\la x+x_\la\r)=t_0+\a_\la\big(v+o(1)\big)\q\text{in }C^\infty_{\mrm{loc}}\l(\R^N\r),$$
where $v$ is a solution to \eqref{soluzione-intera}. When $\O$ is the unit ball instead
$$u_\la(r)=t_0+\a_\la v\l(\frac r{\e_\la}\r)$$
is the explicit solution to \eqref{pfl} if $\a_\la$ is such that $\a_\la v\l(\frac1{\e_\la}\r)=-t_0$.
\end{itemize}
\end{cor}
\begin{rem}$ $\\
Our result applies also to suitable perturbation of $\l((t-t_0)^+\r)^p$, namely $f(t)=\l((t-t_0)^+\r)^p+\l((t-t_0)^+\r)^q$ with $q>p>0$ or when $f$ is given by $(t-t_0)^p\log^2(t-t_0)$ for $t>t_0$. The expansion of $u_\la$ is the same as in $(i)$ or $(ii)$ of Corollary \ref{i3} and $g_0(t)=\l(t^+\r)^p$.\\
\end{rem}

It will be interesting to remove the monotonicity assumption on $f$ at least in the case of an asymptotic linear problem as in the paper \cite{mironescuradulescu}.

\subsection{The case $t_0=-\infty$}\

In this case Theorem \ref{finito} only says that $u_\la\to-\infty$ in $\O$. Our aim is to give a more precise expansion of $u_\la$ and we will see that a crucial role is played by the limit of $f(t)$ as $t\to-\infty$. Let us recall that $f$ is positive and increasing, so the only options are:
\begin{itemize}
\item $\lim_{t\to-\infty}f(t)=c_0>0$
\item $\lim_{t\to-\infty}f(t)=0$
\end{itemize}
Let us consider the first alternative. We have the following

\begin{teo}\label{b2}$ $\\
Let $u_\la$ be the unique negative solution to \eqref{pfl} with $f$ verifying \eqref{ipo} and
$$\lim_{t\to-\infty}f(t)=c_0>0.$$
Then we have that
$$u_\la=\la\big(c_0+o(1)\big)\phi\q\q\q\text{as }\la\to-\infty\text{ in }C^1_\mrm{loc}(\O)$$
where $\phi$ is the solution of the torsion problem 
\begin{equation}\label{tors}
\l\{\begin{array}{ll}-\D\phi=1&\text{in }\O\\\phi=0&\text{on }\pa\O.\end{array}\r.
\end{equation}
\end{teo}
The proof of the previous result is not difficult and it follows by the standard regularity theory. In the other case $\lim_{t\to-\infty}f(t)=0$ interesting new phenomena appear.\\

\begin{teo}\label{infinito}$ $\\
Let $u_\la$ be the unique solution to \eqref{pfl} with $f$ verifying \eqref{ipo} and 
$$\lim_{t\to-\infty}f(t)=0.$$
Assume there exists some $\a(\b),\g(\b)\ge0$ such that $\g(\b)\us{\b\to-\infty}\to0$ and
$$\frac{f(\a(\b)t+\b)}{\g(\b)}\us{\b\to-\infty}\to g(t)\text{ locally uniformly in }t\in\R$$
Then, the following alternative holds:
\begin{itemize}
\item[(i)] If $\frac\b{\a(\b)}\us{\b\to-\infty}\to-\infty$ and in addition $\frac{\g(\b)}{\a(\b)}\us{\b\to-\infty}\to0$ and $\frac{f(\a(\b)t+\b)}{\g(\b)}\ge g_0(t)$ for some $g_0$ satisfying \eqref{keller}, then
$$u_\la=\b_\la+\a_\la\big(v+o(1)\big)\q\q\q\text{as }\la\to-\infty\text{ in }C^1_{\mrm{loc}}(\O),$$
where $v$ is the \emph{large solution} to \eqref{large}, for some $\b_\la\us{\la\to-\infty}\to-\infty,\a_\la\us{\la\to-\infty}\to0$.\\

\item[(ii)] If $\frac\b{\a(\b)}\us{\b\to-\infty}\to A<0$, then
$$u_\la=\a_\la\big(v+o(1)\big)\q\q\q\text{as }\la\to-\infty\text{ in }C^1_{\mrm{loc}}(\O),$$
where $v$ is the unique (negative) solution to 
\begin{equation}\label{solsing}
\l\{\begin{array}{ll}\D v=g(v-A)&\text{in }\O\\v=0&\text{on }\pa\O
\end{array}\r.,
\end{equation}
for some $\a_\la\us{\la\to-\infty}\to+\infty$
\end{itemize}
\end{teo}
Here we observe that also in the case $(ii)$ the solution $v$ to \eqref{solsing} satisfies 
\[\lim_{x\to x_0} \left| v(x)-\psi(d(x,\partial \O))\right|\to 0\]
if $\psi$ is a function which depends only on the nonlinear term $g$, see \cite{bgp} as an example.
\begin{rem}$ $\\
With respect to Theorem \ref{finito}, the statement of Theorem \ref{infinito} has some differences, also because in this case $\b$ cannot be fixed to $t_0$, as the latter equals $+\infty$. Anyway, some simplifications still occur in case $(ii)$.\\
In fact, the limit function $g(t)$ is always a negative power of the type $\frac1{(-t)^p}$ for some $p\ge0$ (see Lemma \ref{omog} for details), and in the case $p=0$ we recover the case of Theorem \ref{b2}. On the other hand, in case $(i)$ other function such as exponentials appear, as explained later on.\\

Moreover, it is not hard to see that one can take $\g(\b)=f(\b)$ (see again Lemma \ref{omog}).\\
Finally, local uniform convergence for $\frac{f(\a(\b)t+\b)}{\g(\b)}$ can actually be assumed only for $t$ for which $\a(\b)t+\b$ is negative (as we evaluate $f$ on $u_\la$ which attains negative values), namely $t<\sup_\b\frac{-\b}{\a(\b)}$.
\end{rem}\

The model nonlinearity of the case $(i)$ in the previous theorem is $f(t)=e^t$. Due to its importance, it seems useful to state explicitly the result.

\begin{cor}$ $\\
Let $\la<0$ and $u_\la$ a family of negative solutions to
$$\l\{\begin{array}{ll}-\D u=\la e^u&\text{in }\O
\\u=0&\text{on }\pa\O.\end{array}\r.$$
Then 
$$u_\la=-\log(-\la)+v+o(1)\text{ in }C^\infty_{\mrm{loc}}(\O)\text{ as }\la\to-\infty$$
and $v$ is the unique positive solution to
\begin{equation}\label{Rob}
\l\{\begin{array}{ll}\D v=e^v&\text{in }\O\\v(x)\us{x\to\pa\O}\to+\infty.\end{array}\r.
\end{equation}
\end{cor}

\begin{rem}$ $\\
If $\O\subset\R^2$ is a simply connected domain then the solution $v$  to \eqref{Rob} satisfies
$$v(x)=2R(x)+\log8$$
where $R$ is the Robin function associated to $\O$, namely the regular part of the Green function computed on the diagonal. Our result, jointly with Suzuki's one (see \cite{s}), gives a complete description for any $\la\in\R$ of the bifurcation diagram containing the minimal branch of \eqref{pfl} with $f(t)=e^t$.\\
Note that our results does not depend on the dimension of the space, differently from the case where $\la>0$ (see \cite{jl} for example).
\end{rem}

\begin{rem}$ $\\
Some example of nonlinearity $f$ where the previous theorem applies are the following:
\begin{itemize}
\item $f(t)=e^{t|t|^{p-1}}$ with $\a_\la=\frac1{p(-\b_\la)^{p-1}}$ and $\b_\la$ verifying $(-\b_\la)^{p-1}e^{\b_\la(-\b_\la)^{p-1}}=-\frac1{p\la}$.\\
So we get $u_\la=\frac{v_\la}{(-\b_\la)^{p-1}}+\b_\la$.
\item $f(t)=(1+|t|)^pe^t$ with $\a_\la=1$ and $\b_\la$ verifying $(-\b_\la)^pe^{\b_\la}=-\frac1\la$.\\
So we get $u_\la=v_\la+\b_\la$.
\end{itemize}
\end{rem}

The paper is organized as follows: in Section \ref{s1} we prove Theorem \ref{casogen}. In Section \ref{s2} we discuss the case $t_0\in\R$ and prove Theorem \ref{finito} and in Section \ref{s3} we prove Theorem \ref{b2} and \ref{infinito}. At the end of the both sections we give some examples where explicity solutions are provided. Finally in the Appendix we show that our assumptions on the nonlinearity $f$ are quite general.\\

\section{General properties of the solution $u_\la$}\label{s1}\

In this section we prove Theorem \ref{casogen}. We start showing some properties of the solution $u_\la$.
\begin{lemma}\label{esiuni}$ $\\
Assume $f$ satisfies \eqref{ipo}. Then for any $\la<0$, \eqref{pfl} has a unique classical solution which is strictly negative in $\O$.
\end{lemma}

\begin{proof}$ $\\
As in \cite{cr} existence for small $\la$ can be proved applying the implicit function Theorem to $F(\la,u):(-\infty,0]\x C^{2,\a}(\O)\cap C_0(\O)\to C^{0,\a}(\O)$ defined as
$$F(\la,u)=\D u+\la f(u)$$
at its trivial zero $(\la_0,u_0)=(0,0)$, as the linearized operator $F'(0,0):v\mapsto\D v$ is invertible. Here $C^{2,\a}(\O), C^{0,\a}(\O)$ are the usual Holder spaces and $C_0(\O)$ is the subspace of the continuous functions on $\O$ that satisfy $u=0$ on $\pa\O$. As $\la<0$ then $u_\la<0$ by the weak and strong maximum principle.\\
More generally, the branch of solutions can be extended at any $(\la_0,u_0)$ with $\la_0<0$; in fact, the linearized operator is
$$F'(\la_0,u_0):v\mapsto\D v+\la_0f'(u_0)v;$$
therefore, being $\la_0<0$ and $f'(u_0)\ge0$, for any $v\not\eq0$ one has
$$\int_\O(F'(\la_0,u_0)v)v=\int_\O(\D v+\la_0f'(u_0)v)v=\int_\O\l(-|\n v|^2+\la_0f'(u_0)v^2\r)<0.$$
This proves the injectivity of $F'(\la_0,u_0)$ and hence that the branch of solutions is a regular curve in a neighborhood of any of its point.\\

Finally, let us show that such a branch exists for every $\la\le0$. Set $I_\la:=\{\la<0\text{ such that }F(\la,u_\la)=0\}$ and $\la^*=\inf I_\la$.\\
We want to show that if $\la^*\in\R$ then $\la^*$ is a minimum for $I_\la$ which contradicts the definition of $\la^*$ since we have already proved that the branch of solutions can be extended from any of its point.\\
By definition there exists a sequence $(\la_n,u_n)$ with $\la_n\to\la^*$ such that $F(\la_n,u_n)=0$. By the maximum principle, any solution $u$ to \eqref{pfl} is not positive, therefore $f(u_n)\le f(0)$ and $(\la^*-\d)f(0)\le-\D u_n\le0$ for some $\d>0$ when $n$ is large enough; hence, by standard elliptic estimates, $u_n$ converges in $C^{2,\a}(\O)$ to some $u^*\in C_0(\O)$ satisfying $F(\la^*,u^*)=0$. This proves existence for any $\la$.\\

To show the uniqueness of the solution for $\la<0$, take two solutions $u,v$ to \eqref{pfl} and consider their difference $u-v$: it solves
$$\l\{\begin{array}{ll}-\D(u-v)=\la(f(u)-f(v))&\text{in }\O\\u-v=0&\text{on }\pa\O\end{array}\r..$$
By testing this equation versus $u-v$, we get
$$\int_\O|\n(u-v)|^2=\int_\O(-\D(u-v))(u-v)=\la\int_\O(f(u)-f(v))(u-v).$$
Since we are taking a non-decreasing $f$, we have $(f(u)-f(v))(u-v)\ge0$, therefore we get $\int_\O|\n(u-v)|^2\le0$, which is possible only if $u\eq v$. This proves the uniqueness and concludes the proof.
\end{proof}\

Next lemma is a comparison principle which is well known as $\la\ge0$. On the other hand the same proof holds for $\la<0$ as well.

\begin{lemma}\label{confr}$ $\\
Let $\O_1\sub\O_2$ and  $u_1,u_2$ be solutions to
$$-\D u_i=g_i(u_i) \ \ \text{in }\O_i ,\q\q\q\q\q\q i=1,2,$$
such that $u_2(x)\le u_1(x)$ on $\pa\O_1$, where $g_1,g_2$ are locally Lipschitz functions, $g_1$ nonincreasing, $g_2(t)\le g_1(t)$ for any $t$. Then,
$$u_2(x)\le u_1(x),\q\q\q\q\q\q\fa x\in\O_1.$$
Moreover, either one has $\O_1=\O_2$, $g_1\eq g_2$, $u_2(x)= u_1(x)$ on $\pa\O_1$, or $u_2(x)<u_1(x)$ for any $x\in\O_1$.
\end{lemma}

\begin{proof}$ $\\
The difference $u_1-u_2$ solves
$$\l\{\begin{array}{ll}-\D(u_1-u_2)=g_1(u_1)-g_2(u_2)&\text{in }\O_1\\u_1-u_2\ge0&\text{on }\pa\O_1\end{array}\r..$$
By writing 
$$g_1(u_1)-g_2(u_2)=\frac{g_1(u_1)-g_1(u_2)}{u_1-u_2}(u_1-u_2)+g_1(u_2)-g_2(u_2),$$
since $g_1(u_2)\ge g_2(u_2)$, then $u_1-u_2$ also satisfies $-\D(u_1-u_2)+c(x)(u_1-u_2)\ge0$ with $c(x)=-\frac{g_1(u_1)-g_1(u_2)}{u_1-u_2}\ge0$ by the monotonicity of $g_1$. Therefore the weak and strong maximum principle gives $u_1-u_2\ge0$, with the strict inequality unless $\O_i, g_i,{u_i}_{|\pa\O_i}$ all coincide.
\end{proof}

\begin{proof}[Proof of Theorem \ref{casogen}]$ $\\
Existence and uniqueness of a negative solution follows from Lemma \ref{esiuni}.\\
The stability of the solution $u_\la$ is an easy consequence of the fact that $\la$ is negative and $f$ non-decreasing.
Moreover, applying Lemma \ref{confr} with $\O_1=\O_2=\O$ and $g_i=\la_i f$ 
we get $u_{\la_2}(x)<u_{\la_1}(x)$ for any $x\in\O$ and $\la_2<\la_1$ and therefore the monotonicity in $\la$ gives the existence of a pointwise limit $u_0(x)=\lim_{\la\to-\infty}u_\la(x)$.\\
We are left with showing that such a limit equals $t_0$ for all $x$; since the monotonicity of $u_\la$ is strict, this would give the inequality $u_\la>t_0$.\\

Let us start with the case when $\O=B_R$ is any ball, whose center is omitted for simplicity. As $u$ has constant sign, the Gidas-Ni-Nirenberg Theorem \cite{gnn} gives that $u$ is radial and radially increasing.\\
We first show that $u_0\ge t_0$ in the case $t_0\in\R$, whereas if $t_0=-\infty$ it is trivial. By contradiction, we assume $u_\la(x)<t_0$ for $x\in B_{R_\la}$ and $u_\la(x)=t_0$
for $x\in\pa B_{R_\la}$, for some $\la<0$ and $R_\la\in(0,R)$. Therefore, $u_\la$ solves 
$-\D u_\la=\la f(u_\la)$ in $B_{R_\la}$ but, since $f(t)=0$ for $t\le t_0$, $u\eq t_0$ also solves the same equation in $B_{R_\la}$ and the solution is unique in view of Lemma \ref{esiuni}; hence, $u_\la\eq t_0$ on $B_{R_\la}$ and we found a contradiction.\\
Now we prove $u_0\le t_0$, which jointly with the previous inequality gives $u_0\equiv t_0$ in $B_R$. If not, $u_\la\ge t_1>t_0$ on some $B_{R_1}\sub B_R$ for any $\la<0$; the monotonicity of $f$ yields $-\D u_\la=\la f(u_\la)\le\la f(t_1)$, and clearly $u_\la\le0$ on $\pa B_{R_1}$. Therefore we may apply the comparison principle to $u_\la$ and $\la f(t_1)\phi$, with $\phi$ being the unique solution to \eqref{tors} in $B_{R_1}$, to get $u_\la\le\la f(t_1)\phi$: since $f(t_1)>0$, we get $u_\la\us{\la\to-\infty}\to-\infty$ a.e. on $B_{R_1}$, contradicting $u_\la\ge t_1$.\\
Finally, the convergence is locally uniform in $B_{R_1}$ because, since $u_\la$ is radially increasing, for any $r<R$ one has
$$\sup_{B_r}|u_\la-t_0|=u_\la(r)-t_0\us{\la\to-\infty}\to0,$$
and when $t_0=-\infty$ 
$$\sup_{B_r}u_\la=u_\la(r)\us{\la\to-\infty}\to-\infty.$$\

Now, let us consider a generic domain $\O$. We consider two balls $B_{R_1}\sub\O\sub B_{R_2}$ and the solutions $u_{i,\la}$ to \eqref{pfl} on $B_{R_i}$: by applying twice Lemma \ref{confr} with $g_1=g_2=\la f$ we get $u_{2,\la}\le u_\la\le u_{1,\la}$ on $B_{R_1}$. Since we already proved that $u_{i,\la}\us{\la\to-\infty}\to t_0$ in $L^\infty_{\mrm{loc}}\l(B_{R_i}\r)$ for both $i$'s, we deduce $u_\la\us{\la\to-\infty}\to t_0$ in $L^\infty_{\mrm{loc}}(B_{R_1})$ and, since the choice of $B_{R_1}$ is arbitrary, also in $L^\infty_{\mrm{loc}}(\O)$.
\end{proof}

In the case when $t_0\in\R$ we can improve Theorem \ref{casogen} getting the following result:
\begin{prop}\label{lem-conv}$ $\\
Let $u_\la$ be the unique solution to \eqref{pfl} for $\la<0$. Assume $f$ satisfies \eqref{ipo}
and that $t_0\in\R$. Then
\begin{equation}\label{stima-che-serve}
\la f(u_\la)\us{\la\to-\infty}\to0\text{ a.e. in }\O,
\end{equation}
and $u_\la\us{\la\to-\infty}\to t_0$ in $C^1_\mrm{loc}(\O)$.
\end{prop}

\begin{proof}
First we prove that for every compact set $K\subset\O$ there exists a constant $C_K$ such that 
\begin{equation}\label{stima-che-serve-2}
\sup_K|\la f(u_\la)|\le C_K.
\end{equation}
By contradiction let us assume that \eqref{stima-che-serve-2} does not hold. Then there exists points $x_\la\in K$ such that $\la f(u_\la(x_\la))\to-\infty$ and, up to a sub-sequence $x_\la\to\ol x\in K$.\\
We take $B_r$ be a ball centered in $\ol x$ and such that $B_{2r}\subset\O$. We call $u_{1,\la}$ the radial solution to \eqref{pfl} in $B_{2r}$. We know by the proof of Theorem \ref{casogen} that $u_{1,\la}$ is radially increasing and that $u_\la<u_{1,\la}$ in $B_{2r}$. The monotonicity of $f$ then gives 
$$\la f(u_\la(x))>\la f(u_{1,\la}(x)) $$
in $B_{2r}$ and since $x_\la\in B_{2r}$ for $\la$ large enough then $\la f(u_{1,\la}(x_\la))\us{\la\to-\infty}\to-\infty$. By the monotonicity of $u_{1,\la}$ we also have that
$$\la f(u_{1,\la}(x))<\la f(u_{1,\la}(x_\la))\to-\infty$$
for every $x\in B_{2r}$ such that $|x-\ol x|>|x_\la-\ol x|.$ In particular we have that, denoting by $A_r=B_r\setminus B_{\frac r2}$, $\la f(u_{1,\la}(x))\to-\infty$ in $A_r$.\\
Last step is to show that this cannot happen. Let $M>0$. There exists $\ol\la<0$ such that 
$$\la f(u_{1,\la}(x))<-M \ \ \text{ in }A_r \ \ \text{ for every }\la<\ol\la.$$ 
We let $z_M$ be the solution to $-\D z_M=-M$ in $A_r$ with Dirichlet boundary conditions. Then by the weak and strong maximum principle we have $u_{1,\la}<z_M$ in $A_r$ and $z_M=-M\phi$ where $\phi$ is the unique solution to \eqref{tors} in $A_r$. Since $M$ is arbitrary this gives a contradiction with 
$$t_0<u_{1,\la}(x)<-M\phi(x)$$
which proves \eqref{stima-che-serve-2}. In order to show \eqref{stima-che-serve} remark that the r.h.s. of the equation satisfied by $u_\la$ is uniformly bounded in every compact set $K$ of $\O$. The standard regularity theory then say that $u_\la-t_0$ is uniformly bounded in $W^{2,p}(K)$ per every $p$, and that, up to a sub-sequence $u_\la-t_0\us{\la\to-\infty}\to0$ in $C^1(K)$. By the weak formulation of \eqref{pfl} we then get that $\la f(u_\la)\us{\la\to-\infty}\to0$ a.e. in $\O$.
\end{proof}

\section{Second order expansion of the solution $u_\la$: the case $t_0\in\R$}\label{s2}\

The aim of this section is to improve estimate \eqref{b1} in Theorem \ref{finito}.

\begin{proof}[Proof of Theorem \ref{finito}]$ $
\begin{itemize}
\item[(i)] By the assumptions on $\a,\g(\a)$ we have $-\frac\a{\g(\a)}\us{\a\searrow0}\to-\infty$. Therefore, after a re-labeling, the ratio $-\frac\a{\g(\a)}$ will decrease monotonically and, for any $\la\ll0$, there will be some $\a_\la$ such that $\la=-\frac{\a_\la}{\g(\a_\la)}$.\\
By such a choice, the function $v_\la$ defined by 
$$v_\la=\frac{u_\la-t_0}{\a_\la}$$
will solve
\begin{equation}\label{vla}
\l\{\begin{array}{ll}\D v_\la=\frac{f(\a_\la v_\la+t_0)}{\g(\a_\la)}&\text{in }\O\\v_\la=-\frac{t_0}{\a_\la}&\text{on }\pa\O\end{array}\r.
\end{equation}
and, by construction, satisfies $v_\la>0$. Next we show that $v_\la$ is bounded from above.\\
Observe that it is not restrictive to assume that $g_0\ge0$ and non-decreasing in such a way \eqref{keller} still holds. In view of the assumption $\frac{f(\a t+t_0)}{\g(\a)}\ge g_0(t)$, we can use Lemma \ref{confr} to get $v_\la\le v_{0,\la}$, with the latter solving 
$$\l\{\begin{array}{ll}\D v_{0,\la}=g_0(v_{0,\la})&\text{in }\O\\v_{0,\la}=-\frac{t_0}{\a_\la}&\text{on }\pa\O.\end{array}\r.$$
Let us introduce the large solution $v_0$ which satisfies
\begin{equation}\label{v0}
\l\{\begin{array}{ll}\D v_0=g_0(v_0)&\text{in }\O\\v_0(x)\us{x\to\pa\O}\to+\infty.\end{array}\r.
\end{equation}
We have $v_{0,\la}\le v_0$ in $\O$ (and, actually $v_{0,\la}\us{\la\to-\infty}\to v_0$) and the boundedness of $v_0$ gives that $v_\la$ is uniformly bounded from above in $L^\infty_{\mrm{loc}}(\O)$. The boundedness of $v_0$ on compact sets of $\O$ follows comparing $v_0$ with the large solution $v_{0,\rho}$ to \eqref{v0} in a small ball $B_\rho$ centered in $x_0\in\O$ and contained in $\O$. By Lemma \ref{confr} $0<v_0<v_{0,\rho}$ and $v_{0,\rho}$ is strictly increasing in the radial variable. This implies that $v_0$ is bounded in the ball $B_{\frac\rho2}$.\\
Since $v_\la$ is bounded in $L^\infty_{\mrm{loc}}(\O)$ and $\D v_\la$ is uniformly bounded for bounded $v_\la$, it will converge in $C^1_{\mrm{loc}}(\O)$ to some function $v$, and in view of the limit \eqref{g}, $v$ will solve $\D v=v^p$. Last step is to prove that $v(x)\us{x\to\pa\O}\to+\infty$. 
Define $\wt g(t):=\sup_{\la<0}\frac{f(\a_\la t+t_0)}{\g(\a_\la)}$ and because the latter converges for any $t$, we have $\wt g(t)<+\infty$ for any $t$. Therefore one may define $\wt v_\la$ as the solution to
\begin{equation}\label{vlatilde}
\l\{\begin{array}{ll}\D\wt v_\la=\wt g\l(\wt v_\la\r)&\text{in }\O\\\wt v_\la=-\frac{t_0}{\a_\la}&\text{on }\pa\O\end{array}\r.,
\end{equation}
and since $\wt g$ is non-negative and non-decreasing, arguing as in Lemma \ref{esiuni} 
we get $v_\la\ge\wt v_\la$, and since $-\frac{t_0}{\a_\la}\us{\la\to-\infty}\to+\infty$, we conclude that $v(x)\us{x\to\pa\O}\to+\infty$, namely $v$ is indeed a large solution.

\item[(ii)] We first consider the case where $\O$ is any domain of $\R^2$ or a convex domain of $\R^N$ with $N\ge3$ and \eqref{i2} holds. Let us take an absolute minimum point $x_\la\in\O$ and set $\a_\la:=\min u_\la-t_0=u_\la(x_\la)-t_0$; due to Theorem \ref{casogen}, we have $\a_\la\us{\la\to-\infty}\to0$ and $\lim_{\la\to-\infty}x_\la=x_0\in\O$ by Remark \ref{gnn}.
Now, since we assume $\frac{\g(\a)}\a\us{\a\searrow0}{\not\to}0$, then we set $\e_\la:=\sqrt\frac{\a_\la}{-\la\g(\a_\la)}\us{\la\to-\infty}\to0$.\\
The rescaled function $v_\la$ defined by 
$$v_\la(x)=\frac{u_\la\l(\e_\la x+x_\la\r)-t_0}{\a_\la}$$
solves
$$\l\{\begin{array}{ll}\D v_\la=\frac{f(\a_\la v_\la+t_0)}{\g(\a_\la)}&\text{in }\frac{\O-x_\la}{\e_\la}\\v_\la(x)\ge v_\la(0)=1\end{array}\r..$$
where $\frac{\O-x_\la}{\e_\la}\to\R^N$ by Remark \ref{gnn}.\\
We have that $v_\la$ satisfies 
\begin{equation}\label{linear}
-\D v_\la+c(x)v_\la=0 
\end{equation}
with $c(x)=\frac{f(\a_\la v_\la+t_0)}{\g(\a_\la)v_\la}\le Cv_\la^{q-1}$ with $0\le q\le1$ by \eqref{i3}.\\
Since $v_\la\ge1$ we get that $|c(x)|\le C$. The Harnack inequality applied to \eqref{linear} in any ball $B_R$ then gives
$$\sup_{B_R} v_\la<C_H\inf_{B_R}v_\la=C_H$$
and so $v_\la$ is uniformly bounded on every compact set of $\R^N$. Using \eqref{g} we can pass to the limit getting that $v_\la\to v$ in $C^1_\mrm{loc}\l(\R^N\r)$ where $v$ is a weak solution to \eqref{soluzione-intera} concluding the proof under the assumption \eqref{i2}.\\
Let us prove the convergence of $v_\la$ in the case $\O=B_R$ is a ball; in this case $v_\la$ is radial and solves
\begin{equation}\label{linear2}
\l\{\begin{array}{ll}v_\la''(r)+\frac{N-1}rv_\la'(r)=\frac{f(\a_\la v_\la(r)+t_0)}{\g(\a_\la)}&0<r<\frac R{\e_\la}\\v_\la'(0)=0\\v_\la(0)=1\end{array}\r..
\end{equation}
Because of the uniform convergence to $g$, there exist sequences $M_n$ and $\la_n$ with $M_n\to+\infty$ and $\la_n\to-\infty$ such that 
\begin{equation}\label{***}
\sup_{0<t\le M_n}\l|\frac{f(\a_{\la_n} t+t_0)}{\g(\a_{\la_n})}-g(t)\r|\us{n\to\infty}\to0.\end{equation} 
Therefore a comparison argument gives $v_{\la_n}(r)\le v_0(r)$ as long as $v_{\la_n}(r)\le M_n$, with $v_0$ solving
$$\l\{\begin{array}{ll}v_0''(r)+\frac{N-1}r v_0'(r)=g(v_0(r))+1&r\in\R\\v_0'(0)=0\\v_0(0)=1\end{array}\r..$$
We have that $v_0$ is well-defined for any $r$ because $g$ does not satisfy the condition \eqref{keller} (see Theorem $4$ in \cite{kel}). Now taking $r_n$ such that $v_{\la_n}(r_n)=M_n$, one has $v_0(r_n)\ge M_n\us{n\to\infty}\to+\infty$ and then $r_n\us{n\to\infty}\to+\infty$. Therefore, for any fixed $r>0$ we have, for large $n$, $r\le r_n$ and $v_{\la_n}(r)\le v_0(r)\le C$. Since $v_{\la_n}$ is bounded in $L^\infty_{\mrm{loc}}$ we can pass to the limit in \eqref{linear2} and $v_{\la_n}$ will also converge to the solution $v$ to \eqref{soluzione-intera}. Since from every sequence $\la_n\to-\infty$ we can extract a subsequence $\tilde\la_n$ that satisfies 
\eqref{***} then $v_\la$ converges to the solution $v$.
\end{itemize}
\end{proof}

\begin{rem}$ $\\
In the case of a more general decay of the function $f$, as when \eqref{g2} holds instead of \eqref{g}, we can argue as in the proof of Theorem \ref{finito} $(i)$ choosing $\la=-\frac{\a_\la}{f(\b_\la)}$ and replacing the function $v_\la$ with
$$v_\la:=\frac{u_\la-\b_\la}{\a_\la}.$$
As in the previous case one can find that $v_\la$ is bounded from above since $v_\la<v_0$. To prove that $v_\la$ is bounded by below we define $\wt g(t):=\sup_{t_0<\b<0}\frac{f(\a(\b)t+\b)}{f(\b)}$ and because the latter converges for any $t$, we have $\wt g(t)<+\infty$ for any $t$. Therefore one may define $\wt v$ as the solution to
\begin{equation}\label{vtilde}
\l\{\begin{array}{ll}\D\wt v=\wt g\l(\wt v\r)&\text{in }\O\\\wt v=0&\text{on }\pa\O\end{array}\r..
\end{equation}
Since $\wt g$ is non-negative and non-decreasing, arguing as in Lemma \ref{esiuni} we have that $\wt v$ is uniquely defined and belongs to $L^\infty(\O)$; therefore, Lemma \ref{confr} yields $v_\la\ge\wt v$ on $\O$, namely $v_\la$ is also uniformly bounded from below. The convergence of $v_\la$ to $v$ then follows as in the previous case. The only difference is that $v$ is a large solution to \eqref{large}.
\end{rem}

We end this section with two examples where Theorem \ref{finito} applies. In this case we exhibit explicitly the solutions.
\begin{ese}\label{ese2}$ $
\begin{itemize}
\item[(i)] $f(t)=\l((t+1)^+\r)^\frac{N+2}{N-2}$, $\O=B_1\sub\R^N$, $N\ge3$.\\
We are in the first alternative of Theorem \ref{finito}, with $t_0=-1,\a(\b)=\b+1$, $g_0(t)=g(t)=(t+1)^\frac{N+2}{N-2}$. In this case we have explicit solutions given by
$$u_\la(x)=\l(\frac{\d_\la-1}{\d_\la-|x|^2}\r)^\frac{N-2}2-1\q\q\q\d_\la:=\frac{N^2-2N-2\la+\sqrt{N(N-2)(N(N-2)-4\la)}}{-2\la}.$$
and 
$$u_\la(x)\to-1\text{ in }L^\infty_\mrm{loc}(B_1).$$
Taking $\b_\la:=-1+(-\la)^\frac1{p-1},\a_\la:=(-\la)^\frac1{p-1}$, we have
$$v_\la(x):=\frac{u_	\la+1}{(-\la)^\frac1{p-1}}-1=\l(\frac{ \la^2(\d_\la-1)}{\d_\la-|x|^2}\r)^\frac{N-2}2-1\us{\la\to-\infty}\to\l(\frac{N(N-2)}{1-|x|^2}\r)^\frac{N-2}2-1=:v(x),$$
the latter being the unique large solution to $\D v=(v+1)^\frac{N+2}{N-2}$ in $\O$.

\item[(ii)] $f(t)=(t+1)^+$, $\O=(-1,1)\sub\R$.\\
We are in the second alternative of Theorem \ref{finito}, with $\g(\a)=\a$ and $g(t)=t$. In fact, explicit solutions are given by
$$u_\la(x)=\frac{\cosh\l(\sqrt{-\la}x\r)}{\cosh\sqrt{-\la}}-1$$
and 
$$u_\la(x)\to-1\text{ in }L^\infty_\mrm{loc}(-1,1).$$
Taking $\a_\la=u_\la(0) +1=\frac1{\cosh{\sqrt{-\la}}}$, $x_\la=0$ and $\e_\la=\frac1{\sqrt{-\la}}$, we have
$$v_\la(x)=\frac{u_\la\l(\frac x{\sqrt{-\la}}\r)+1}{u_\la(0)+1}\eq\cosh x:=v(x),$$
the latter being the solution to
$$\l\{\begin{array}{ll}v''(x)=v(x)&x\in\R\\v(0)=1\\v'(0)=0\end{array}\r..$$
In the case of the same nonlinearity on $\O=B_1\sub\R^N$, $N\ge2$, the same argument holds true with $\cosh x$ being replaced with the solution to 
\begin{equation}\label{****}
\l\{\begin{array}lv''(r)+\frac{N-1}rv'(r)=v(r)\\v(0)=1\\v'(0)=0.\end{array}\r.
\end{equation}
Using the last statement in Corollary \ref{i3} we can compute the explicit solutions to \eqref{pfl} for any $\la$, which are given by
$$u_\la (r)=\frac1{v\l(\sqrt{-\la}\r)}v\l(\sqrt{-\la}r\r)-1$$
where $v$ the unique radial solution to \eqref{****}.\\
For $N=3$, since it is known that  $v(r)=\frac{\sinh r}r$, we have that
$$u_\la(r)=\frac1{\sinh\sqrt{-\la}}\frac{\sinh\l(\sqrt{-\la}r\r)}r-1.$$
 A simple computation then gives
\[J_\la(u_\la)=\frac 12 \int_{B_1}|\nabla u_\la|^2-\la (u_\la+1)^2 dx = (2\pi+o(1))\sqrt{-\la}\]
as $\la \to \infty$ if $J_\la$ is as defined in \eqref{J}.
\end{itemize}
\end{ese}

\section{Refined expansions of the solution $u_\la$: the case $t_0=-\infty$}\label{s3}\

In this section we split the proof in two parts, according the limit of $f(t)$ at $-\infty$. Let us start with the case where the limit at $-\infty$ is positive. 

\subsection{The case $\lim_{t\to-\infty}f(t)=c_0>0$}\

In this case \emph{second order estimates} for the solution $u_\la$ will be provided without additional assumptions on $f$.

\begin{proof}[Proof of Theorem \ref{b2}]$ $\\
Denote by $v_\la=\frac{u_\la}\la$. It verifies,
\begin{equation}\label{c4}
\begin{cases}
-\D v_\la=f(u_\la)&\text{in }\O\\v_\la=0&\text{on }\pa\O.
\end{cases}
\end{equation}
By the properties of $f$ we get that
$$c_0\le f(u_\la)\le f(0)$$
and then by the standard regularity theory we get that there exists $\phi$ such that $v_\la\to\phi$ in $C^1(\O)$. Moreover, by Theorem \ref{casogen} and $\lim_{t\to-\infty}f(t)=c_0>0$
we get that
$$f(u_\la)\to c_0\q\text{in }L^\infty_\mrm{loc}(\O).$$
Passing to the limit in \eqref{c4} the claim follows.
\end{proof}
Next we consider the other case.\\

\subsection{The case $\lim_{t\to-\infty}f(t)=0$}\label{s4}\

Here the argument are very similar to that in Theorem \ref{finito}. We will sketch the main points.

\begin{proof}[Proof of Theorem \ref{infinito}]$ $
\begin{itemize}
\item[(i)] Since $\frac{\g(\b)}{\a(\b)}\us{\b\to-\infty}\to0$, without loss of generality we may assume the ratio to decrease monotonically and, for $\la\ll0$, we take $\b_\la,\a_\la=\a(\b_\la)$ such that $\la=-\frac{\a(\b_\la)}{\g(\b_\la)}$.\\
We define $v_\la:=\frac{u_\la-\b_\la}{\a(\b_\la)}$, which will solve
$$\l\{\begin{array}{ll}\D v_\la=\frac{f(\a_\la v_\la+\b_\la)}{\g(\b_\la)}&\text{in }\O\\v_\la=-\frac{\b_\la}{\a_\la}&\text{on }\pa\O\end{array}\r.,$$
and we have the inequalities $\wt v\le v_\la\le v_0$ in $\O$, with $v_0,\wt v$ respectively defined by \eqref{v0}, \eqref{vtilde}. From this we deduce $v_\la$ is locally uniformly bounded in $\O$ and converges to some solution $v$ to $\D v=g(v)$; finally, we take $\wt v_\la$ solving
$$\l\{\begin{array}{ll}\D\wt v_\la=\wt g\l(\wt v_\la\r)&\text{in }\O\\\wt v_\la=-\frac{\b_\la}{\a_\la}&\text{on }\pa\O\end{array}\r.,$$
therefore from the inequality $v_\la\ge\wt v_\la$ and $\frac{-\b_\la}{\a_\la}\us{\la\to-\infty}\to+\infty$ we deduce $v|_{\pa\O}=+\infty$.

\item[(ii)] Since $\frac\b{\a(\b)}\ge A-1$, we must have $\a(\b)\us{\b\to-\infty}\to+\infty$, hence again $\frac{\g(\b)}{\a(\b)}\us{\b\to-\infty}\to0$; we may assume the latter limit to decrease monotonically and take $\b_\la$ so that $-\la=\frac{\a(\b_\la)}{\g(\b_\la)}$. Therefore, $v_\la:=\frac{u_\la}{\a(\b_\la)}$ solves 
$$\l\{\begin{array}{ll}\D v_\la=\frac{f(\a(\b_\la)v_\la)}{\g(\b_\la)}&\text{in }\O\\v_\la=0&\text{on }\pa\O.\end{array}\r.$$\\
By assumption, we have $v_\la\le0$ and moreover $\frac{f(\a(\b_\la)t)}{\g(\b_\la)}\us{\la\to-\infty}\to g(t-A)$
 locally uniformly in $t$, therefore one may define $\wt g(t):=\sup_\la\frac{f(\a(\b_\la)t)}{\g(\b_\la)}<+\infty.$ Since $\wt g(t)\us{t\to0}\to+\infty$ and it increases monotonically for any $t<0$, there exists a solution $\wt v$ to 
$$\l\{\begin{array}{ll}\D\wt v=\wt g\l(\wt v\r)&\text{in }\O\\\wt v=0&\text{on }\pa\O,\end{array}\r.$$
which is uniformly bounded in $\ol\O$ (see Theorem 1.1 in \cite{crt} for details).\\
Therefore, $v_\la\ge\wt v$ hence it is uniformly bounded in $C\l(\ol\O\r)$ and, in view of the convergence of $f$, it will converge to the solution to \eqref{solsing}.
\end{itemize}
\end{proof}

As in the previous section we end with two examples where explicit solutions are provided:

\begin{ese}\label{ese}$ $
\begin{itemize}
\item[(i)] $f(t)=e^t$ on $\O=B_1\sub\R^2$.\\
We are in the first alternative of Theorem \ref{infinito}, with $\a(\b)=1$ and $g(t)=g_0(t)=e^t$. In fact, explicit solutions are given by
$$u_\la(x)=\log\fr{8\d_\la}{-\la\l(\d_\la-|x|^2\r)^2},\q\q\q\q\d_\la=1+\fr{4+2\sqrt{4-2\la}}{-\la}=1+\frac{2\sqrt2+o(1)}{\sqrt{-\la}}.$$
Taking $\b_\la=\log(-\la),\a_\la=1$, we have
$$v_\la(x)=\log\fr{8\d_\la}{\l(\d_\la-|x|^2\r)^2}\us{\la\to-\infty}\to\log\frac8{\l(1-|x|^2\r)^2}=:v(x),$$
the latter being the large solution to $\D v=e^v$ on $\O$. Moreover a straightforward computation gives that
\begin{equation}
J_\la(u_\la):=\frac 12 \int_{\O}|\nabla u_\la|^2 dx-\la \int_\O\left(e^{u_\la}-1\right)dx=\big(2\sqrt2\omega_N+o(1)\big)\sqrt{-\la}
\end{equation}
where $\omega_N$ is the area of the unit ball in $\R^N$.
\item[(ii)] $f(t)=\frac1{(1-v)^3}$ on $\O=(-1,1)\sub\R$.\\
We are in the second alternative of Theorem \ref{infinito}, with $p=3,\g(\a)=\a^3$. In fact, explicit solutions are given by
$$u_\la(x)=1-\sqrt{1+\frac{-2\la}{1+\sqrt{1-4\la}}\l(1-x^2\r)}.$$
Taking $\a_\la=\frac1{(-\la)^\frac14}$, we have
$$v_\la(x)=\frac{u_\la(x)}{(-\la)^\frac14}\us{\la\to-\infty}\to-\sqrt{1-x^2}=:v(x),$$
the latter being the solution to $\l\{\begin{array}{ll}v''=\frac1{(-v)^3}&\text{in }(-1,1)\\v(1)=v(-1)=0\end{array}\r.$.
\end{itemize}
\end{ese}\

\appendix

\section{Appendix}\

In this Appendix we show that the assumption on $f$ considered in Theorems \ref{finito}, \ref{infinito} are rather general. In fact, they occur any time one has the following \emph{asymptotic homogeneity} condition:
\begin{equation}\label{omogcond}
\frac{f(\a(\b)t+\b)}{\g(\b)}\us{\b\searrow t_0}\to g(t)\q\q\q\text{locally uniformly for }-\sup_\b\frac{\b-t_0}{\a(\b)}<t<\sup_\b\frac{-\b}{\a(\b)}\text{, for some }g\not\eq0,
\end{equation}
for some $\a(\b),\g(\b)>0$, with $f$ satisfying \eqref{ipo} and \eqref{t0} defined by $t_0$.\\
On the other hand, condition \eqref{omogcond} seems to be necessary in order to pass to the limit in the equation \eqref{pfl}, after taking a rescalement.

\begin{lemma}\label{omog}$ $\\
Assume $f$ satisfies \eqref{ipo}, $t_0$ is defined by \eqref{t0} and there exist and $\a(\b),\g(\b)>0$ such that \eqref{omogcond}. Then,
\begin{enumerate}
\item In \eqref{omogcond}, one can take without restriction $\g(\b)=f(\b)$ and $g$ satisfying $g(0)=1$;
\item If $t_0\in\R$, then \eqref{omogcond} can hold only of $\a(\b)\us{\b\nearrow t_0}\to0$;
\item If $t_0\in\R$ and $\lim_{\b\searrow t_0}\frac{\b-t_0}{\a(\b)}=\ol t\in\R_{>0}$, then $g(t)=\l(\l(t+\ol t\r)^+\r)^p$ for some $p>0$ and $\frac{f(\a(\b)t+t_0)}{f(\b)}\us{\b\searrow t_0}\to t^p$ locally uniformly in $t>0$;
\item If $t_0=+\infty$ and $\lim_{b\to-\infty}\frac{-\b}{\a(\b)}=-\ol t\in\R_{<0}$, then $g(t)=\frac1{\l(\l(\ol t-t\r)^+\r)^p}$ for some $p\ge0$ and $\frac{f(\a(\b)t)}{f(\b)}\us{\b\searrow t_0}\to\frac1{(-t)^p}$ locally uniformly in $t<0$.
\end{enumerate}
\end{lemma}

\begin{proof}$ $\\
\begin{enumerate}
\item Because of the convergence at $t=0$, one has $\frac{f(\b)}{\g(\b)}\us{\b\searrow t_0}\to g(0)$, therefore $\frac{f(\a(\b)t+\b)}{f(\b)}\us{\b\searrow t_0}\to\frac{g(t)}{g(0)}$.
\item If $t_0\in\R$ and $\a(\b)\ge\d_0>0$ on a sub-sequence, then we would get:
$$g(1)=\lim_{\b\searrow t_0}\frac{f(\a(\b)+\b)}{f(\b)}\ge\frac{f(\d_0+\b)}{f(\b)}.$$
Since $f(\d_0)>f(t_0)=0$, then passing to the limit on the right-hand side we would get $+\infty$, hence a contradiction.
\item Since $\b=t_0+\a(\b)\l(\ol t+o(1)\r)$, then for any $t,\e>0$ we will have, for $\b$ close enough to $t_0$, $\a(\b)\l(t+\ol t-\e\r)+t_0\le\a(\b)t+\b\le\a(\b)\l(t+\ol t+\e\r)+t_0$; therefore, by the monotonicity of $f$,
$$\limsup_{\b\searrow t_0}\frac{f\l(\a(\b)\l(t+\ol t-\e\r)+t_0\r)}{f(\b)}\le g(t)\le\liminf_{\b\searrow t_0}\frac{f(\a(\b)\l(t+\ol t+\e\r)+t_0)}{f(\b)}.$$
As $\e$ is arbitrary and $g$ is continuous, we conclude that $\frac{f\l(\a(\b)\l(t+\ol t\r)+t_0\r)}{f(\b)}\us{\b\searrow t_0}\to g(t)$. Now, we take $\wt\b$ such that $\a\l(\wt\b\r)=2\a(\b)$ and compute the previous limit with $2t+\ol t$ in place of $t$:
$$g\l(2t+\ol t\r)=\lim_{\b\searrow t_0}\frac{f\l(\a(\b)\l(2t+2\ol t\r)+t_0\r)}{f(\b)}=\lim_{\b\searrow t_0}\frac{f\l(2\a(\b)\l(t+\ol t\r)+t_0\r)}{f(\b)}=\lim_{\b\searrow t_0}\frac{f\l(\a\l(\wt\b\r)\l(t+\ol t\r)+t_0\r)}{f\l(\wt\b\r)}\frac{f\l(\wt\b\r)}{f(\b)}.$$
Since $\a\l(\wt\b\r)=2\a(\b)\to0$, we have $\wt\b\to t_0$, hence the first factor of the right-hand side goes to $g(t)$; therefore we get, for any $t$, we have $g\l(2t+\ol t\r)=Lg(t)$, with $L:=\lim\frac{f\l(\wt\b\r)}{f(\b)}$. Because of the uniform convergence, $g$ is continuous, and it is also non-negative, non-decreasing and satisfies the condition $g\l(2t+\ol t\r)=Lg(t)$: it must be of the kind $g(t)=C\l(\l(t+\ol t\r)^+\r)^p$ for some $C>0,p\ge0$, and $g(0)=0$ implies $C=1$. The final limit follows by passing from $t+\ol t$ to $t$.
\item We argue similarly as before. Since $\b=\a(\b)\l(-\ol t+o(1)\r)$, then $\frac{f\l(\a(\b)\l(t-\ol t\r)+t_0\r)}{f(\b)}\us{\b\searrow t_0}\to g(t)$. We take $\wt\b$ such that $\a\l(\wt\b\r)=2\a(\b)$, and in this case we have $\wt\b\to-\infty$ because the latter goes to $+\infty$. Therefore, as before,
$$g\l(2t-\ol t\r)=\lim_{\b\to-\infty}\frac{f\l(2\a(\b)\l(t-\ol t\r)\r)}{f(\b)}=\lim_{\b\searrow t_0}\frac{f\l(\a\l(\wt\b\r)\l(t-\ol t\r)\r)}{f\l(\wt\b\r)}\frac{f\l(\wt\b\r)}{f(\b)}=Lg(t),$$
which implies $g$ is of a power type and the rest of the statement.
\end{enumerate}
\end{proof}

\bibliography{fralucamax}
\bibliographystyle{abbrv}

\end{document}